\documentclass[11pt]{amsart}
\usepackage{times}
\usepackage[T1]{fontenc}
\usepackage{amssymb, amsthm, amsmath}
\usepackage[active]{srcltx}

\usepackage{setspace}
\usepackage{geometry}
\usepackage{showlabels}
\usepackage{hyperref}

\DeclareMathOperator{\acl}{acl}
\DeclareMathOperator{\dcl}{dcl} 
 
 \DeclareMathOperator{\id}{id}

 \DeclareMathOperator{\dom}{dom}

\DeclareMathOperator{\tp}{tp}

\newtheorem{introtheorem}{Theorem}

\newtheorem{theorem}{Theorem}[section]

\newtheorem{claim}{Claim}[theorem]

\newtheorem{corollary}[theorem]{Corollary}

\newtheorem{fact}[theorem]{Fact}
\newtheorem{lemma}[theorem]{Lemma}

\newtheorem{proposition}[theorem]{Proposition}

\theoremstyle{definition}
\newtheorem{definition}[theorem]{Definition}

\newtheorem{remark}[theorem]{Remark}

\newcommand{\Nn}{{\mathbb{N}}}
\newcommand{\Qq}{{\mathbb{Q}}}

\newcommand{\Zz}{{\mathbb {Z}}}

\newcommand{\CD}{{\mathcal D}}
\newcommand{\CL}{{\mathcal L}}
\newcommand{\CK}{{\mathcal K}}

\newcommand{\CM}{{\mathcal M}}

\newcommand{\CO}{{\mathcal O}}

\newcommand{\CF}{{\mathcal F}}

\newcommand{\0}{\emptyset}

\renewcommand{\phi}{\varphi}

\newenvironment{claimproof}[1][\proofname]
  {%
    \proof[#1]%
  }
  {%
    \endproof%
  }

\newcommand{\dclindep}[1][]{%
  \mathrel{
    \mathop{
      \vcenter{
        \hbox{\oalign{\noalign{\kern-.3ex}\hfil$\vert$\hfil\cr
              \noalign{\kern-.7ex}
              $\smile$\cr\noalign{\kern-.3ex}}}
      }
    }\displaylimits_{#1}
  }
}

\long\def\symbolfootnote[#1]#2{\begingroup%
\def\thefootnote{\fnsymbol{footnote}}\footnote[#1]{#2}\endgroup}

\makeatletter

\def\Ind#1#2{#1\setbox0=\hbox{$#1x$}\kern\wd0\hbox to 0pt{\hss$#1\mid$\hss}
\lower.9\ht0\hbox to 0pt{\hss$#1\smile$\hss}\kern\wd0}

\def\Notind#1#2{#1\setbox0=\hbox{$#1x$}\kern\wd0\hbox to 0pt{\mathchardef
\nn=12854\hss$#1\nn$\kern1.4\wd0\hss}\hbox to
0pt{\hss$#1\mid$\hss}\lower.9\ht0 \hbox to
0pt{\hss$#1\smile$\hss}\kern\wd0}

\def\la{\langle}
\def\ra{\rangle}
\def\qp{\mathbb Q_p}

\def\dpr{\mathrm{dp\text{-}rk}}
\def\sub{\subseteq}

\makeatother

\title{Fields interpretable in $P$-minimal fields}

\date{\today}
\author{Yatir Halevi}
\address{Department of Mathematics, Ben Gurion University of the Negev, Be'er-Sheva 84105, Israel}
\email{yatibe@post.bgu.ac.il}

\author{Assaf Hasson}
\address{Department of Mathematics, Ben Gurion University of the Negev, Be'er-Sheva 84105, Israel}
\email{hassonas@math.bgu.ac.il}

\author{Ya'acov Peterzil}
\address{Department of Mathematics, University of Haifa, Haifa, Israel}
\email{kobi@math.haifa.ac.il}

\thanks{The first author would like to thanks the Israel Science Foundation for its support of this research (grant No. 181/16) and the Kreitman foundation fellowship. The third author was partially supported by Israel Science Foundation grant number 290/19 }

\begin{document}

\begin{abstract}
    We prove that an infinite field interpretable in a $p$-adically closed field $K$ is  definably isomorphic to a finite extension of $K$. The result remains true in any $P$-minimal field where definable functions are generically differentiable.
\end{abstract}

\maketitle

In \cite{HaPetRCVF}, interpretable fields in various expansions of real closed valued fields were classified using analysis of one dimensional quotients which are definably embedded in the field.
Here we use similar methods to classify interpretable fields in various expansions of $p$-adically closed fields.

In \cite{PilQp} Pillay proved that any field definable in $\qp$ is definably isomorphic to a finite extension of $\qp$, and asked whether the same result is true of interpretable fields. Our main theorem gives a positive answer to Pillay's question.

\begin{introtheorem}
    Let $K$ be a $P$-minimal field. If, additionally, definable functions in $K$ are generically differentiable, then any infinite field interpretable in $K$ is definably isomorphic to a finite extension of $K$. In particular, the result holds for $p$-adically closed fields.
\end{introtheorem}

The proof  does not use elimination of imaginaries (that is not known, in general, for $P$-minimal fields). It uses implicitly the fact that $P$-minimal fields are uniformities (see \cite{SimWal} for details). The results of Sections 4-5 do  not use $P$-minimality beyond those consequences that are treated axiomatically in \cite{SimWal} (and even \cite{DolGodViscerality}). For that reason we expect that our methods can also be applied outside of the $P$-minimal context.

Note that recently Walsberg showed, \cite{WalTrace}, that a positive answer to the above question of Pillay's  would provide an example of an NIP theory not interpreting an infinite field whose Shelah expansion does. We mention  also that a positive answer to Pillay’s question on interpretable fields in $\qp$ was announced also by E. Alouf, A. Fornasiero and J. de la Nuez Gonzalez.

\vspace{.2cm}

\noindent{\bf Acknowledgements} We thank Immanuel Halupczok and Dugald Macpherson for several discussions during the preparation of the article.

\section{Background and preliminaries}

The notion of $P$-minimal fields was introduced by Haskell and Macpherson in \cite{HasMac}. Recall that a valued field $K$ is \emph{$P$-minimal} if it is $p$-valued, its value group is a $\Zz$-group and in every structure elementary equivalent to $K$ every definable subset of $K$ is quantifier-free definable in the Macintyre language for valued fields. We denote its value group by $\Gamma$ and its valuation ring by $\CO$. It is one of the main results of \cite{HasMac} that $P$-minimal fields are henselian, and therefore, as pure valued fields $p$-adically closed.

There are various cell decomposition results for $P$-minimal fields, but we will only explicitly use the fact, following from the very definition, that definable subsets of $K$ itself can be partitioned into finite many cells, as in the case of $p$-adically closed fields. I.e. every definable set is a disjunction of sets of the form
\[
\{x\in K: \gamma_1< v(x)<\gamma_2 \wedge P_n(\lambda\cdot x)\},
\]
where $P_n$ is the $n$-th power predicate, $\lambda\in K$ and $\gamma_1,\gamma_2\in \Gamma\cup\{\infty,-\infty\}$ and $n\in\mathbb{N}$.

In particular, every infinite  definable subset of $K$ has non-empty interior with respect to the valuation topology (this also follows from Simon's \cite{SimDPMinOrd}). Since any $P$-minimal field is dp-minimal, this shows that $P$-minimal fields satisfy the assumptions of Simon and Walsberg's tameness conditions for topological dp-minimal structures, \cite{SimWal}. Most of the topological properties of $P$-minimal fields we will be using below follow from this work of Simon and Walsberg, though -- in most cases -- specialised to the case of $P$-minimal fields they were already known by earlier work.

In particular, this implies that $P$-minimal fields eliminate $\exists^{\infty}$ (see \cite[Lemma 2.1]{SimWal}) in the valued field sort. We will use this fact in several places without explicit reference. We note however, that this property does not transfer to imaginary sorts (such as the value group and $K/\CO$). We will also be using the fact that $\acl$ (the model theoretic algebraic closure) satisfies the Steiniz Exchange principle in $P$-minimal fields (\cite[Corollary 6.2]{HasMac}). As a consequence, in the valued field sort of $K$, dp-rank, topological dimension and $\acl$-dimension coincide (see \cite[Proposition 2.4]{SimWal}) and dp-rank is additive. Most of the result we need from \cite{SimWal} also appear in the the work of Dolich and Goodrick on viscerality, \cite{DolGodViscerality}, and may be relevant for possible generalisations of the present work into structures with IP. 

%\subsection{Differentiation}\label{SS:different}
%Since the mean value theorem does not hold in general for valued fields, a stronger notion of differentiability is needed. 

%\begin{definition}
%Let $(K,v)$ be  a valued field, $f:U\to K^m$ a map, where $U\subseteq K^n$ is an open subset. We say that $f$ is \emph{strictly differentiable} in $x_0\in U$ if there exists a continuous linear map $D_{x_0}f:K^n\to K^m$ such that for any $r\in \Gamma$ there is an open ball $x_0\in B\subseteq U$ such that 
%\[\widehat v\left( f(x)-f(x')-D_{x_0}f(x-x')\right)>r\cdot \widehat v( x-x')\]
%for any $x,x'\in B$ (where $\widehat v(a_1,\dots,a_n)=\min\{v(a_i)\}$. In this case we say that $f$ is %$\CS^1$ in $x_0$.

%If $f$ strictly differentiable in every point of $U$ we say that $f$ is $\CS^1$.
%\end{definition}

%It is not hard to see that $f$ is strictly differentiable in $x_0$ then it is also differentiable in $x_0$, in particular $\CC^1$, and that $D_{x_0}f$ is exactly the total derivative of $f$ at $x_0$. The chain rule for differentiable functions works the same in this non-archimedean setting. See \cite[Section 1.4]{p-adiclie} (the rank $1$ assumption in that section is immaterial \textbf{i think - especially the place were uniqueness of total derivative is proved}).

\section{Strong internality to $K$}

Our aim in this section is, given an interpretable field $\CF$, to  find a dp-minimal $S\subseteq \CF$ that is in definable bijection with a subset of  $K$. We proceed in three steps, we first reduce the problem to subsets of $\CF$ that are in definable bijection with an infinite subset of  $K/E$, where $E$ is some definable equivalence relation on $K$. We then show that for any interpretable set $S$ of the form $K/E$, there exists a finite-to-one (partial) function $f:S\to V$ with infinite domain, where $V$ is either $K$, $\Gamma$ or $K/\CO$. Finally we show that if $S\subseteq \CF$ then this set can only be $K$. \\

\noindent We begin with several general results about definable equivalence relations in $P$-minimal fields.

\subsection{Definable one-dimensional quotients in $P$-minimal fields}

We assume that $K$ is a $P$-minimal field and $E$ a definable equivalence relation on $K$ with infinitely many classes.
 Our aim is to find a definable function with finite fibres from an infinite subset of $K/E$ into either $\Gamma$, $K/\CO$ or $K$ itself.
  We need some preperation. For $\gamma\in \Gamma$ and $a\in K$, we let
  $$B_\gamma(a)=\{x\in K:v(x-a)>\gamma\},$$ and call $\gamma$ {\em the radius of the $\gamma$-ball} $B_{\gamma}(a)$. 

By using the fact that any $P$-minimal valued field is elementary equivalent to a $p$-adic field, it is not hard to verify the following.

\begin{lemma}\label{L:ball in elem exten of p-adics}
\begin{enumerate}
\item For any $\gamma\in \Gamma$ and $n\in \mathbb N$, any $\gamma$-ball contains only finitely many $\gamma+n$-balls.
\item If a $\gamma$-ball is covered by finitely many balls with radii in $\mathbb Z$ then $\gamma\in \mathbb Z$.
\end{enumerate}
\end{lemma}

The following is the key lemma for this part of the argument. We thank D. Macpherson for sketching the proof for us.
We first recall \cite[Lemma 2.3]{HasMac}:
\begin{fact}\label{HM} Let $K$ be a $p$-adically closed field and let $n\in \mathbb N$ with $n>1$ and $x,y,a\in K$. Suppose that $v(y-x)>2v(n)+v(y-a)$. Then $x-a, y-a$ are in the same coset of $P_n$.
\end{fact}

\begin{lemma}\label{L:dugald}
Assume that $X\subseteq K$ is a definable set, $\gamma_0\in \Gamma$ and $X$ contains infinitely many $\gamma_0$-balls.
Then for every $k\in \mathbb N$, $X$ contains a ball of radius $\gamma_0-k$. 

In particular, if $K$ is $\aleph_0$-saturated then $X$ contains a ball of radius $\gamma$ satisfying $\gamma<\gamma_0-k$ for all $k\in \mathbb{N}.$
\end{lemma}
\begin{proof} We assume that $X$ contains infinitely many $\gamma_0$-balls. 
%We may move to an elementary extension of $K$ and prove the first part of the lemma then transfer it back to $K$. Thus, we may assume that $K$ is $\aleph_0$-saturated. 

Partitioning $X$  into cells, and, if needed, translating by an element of $K$, we may assume that $X$ has the form
\[
\{x\in K: \gamma_1< v(x)<\gamma_2 \wedge P_n(\lambda\cdot x)\},
\]
where $\lambda\in K$, $\gamma_1<\gamma_2\in \Gamma\cup\{\infty,-\infty\}$ and $n\in\mathbb{N}$. Multiplying $X$  by a scalar,  we may assume that $\gamma_0=0$. 

By Lemma \ref{L:ball in elem exten of p-adics}, $X$ is not contained in a ball $B_{-m}(0)$ for any $m\in \Nn$. So for every $k\in \mathbb N$,  there is some $y_0\in X$  such that $v(y_0)<-k$. 
%Indeed, if not then $X$ is contained in some ball $B_n(0)$, for $n\in \mathbb Z$, so by Lemma \ref{L:ball in elem exten of p-adics}, it would not contains infinitely many $0$-balls.
We can thus fix   $y_0\in X$, such that $v(y_0)+2v(n)<-1$.
We claim that $B_{-1}(y_0)\sub X$.

Indeed, assume that $v(x-y_0)>-1$. Then, since $v(y_0)<-1$, we have $v(x)=v(y_0)$ and therefore  $\gamma_1< v(x)=v(y_0)<\gamma_2$. Thus it is sufficient to see that $x$ and $y_0$ are in the same $P_n$-coset.

By our choice of $y_0$ we have $v(x-y_0)>-1>v(y_0)+2v(n)$, hence by Fact \ref{HM} (with $a=0$ there), $x$ and $y_0$ are in the same $P_n$-coset, so $x\in X$. Thus $B_{-1}(y_0)\sub X$.

We have thus shown that $X$ contains at least one ball of radius $-1$. After removing from $X$ a single ball of radius $-1$ it still contains infinitely many $0$-balls so we can find in $X$ a second ball of radius $-1$. Continuing in this manner we find infinitely many balls in $X$ of radius $-1$.

It follows that if $X$ contains infinitely many balls of radius $\gamma_0$ then it also contains infinitely many balls of radius $\gamma_0-1$. Repeating the process we obtain in $X$ infinitely many balls of radius $\gamma_0-k$, for every $k\in \mathbb N$.

If $K$ is $\aleph_0$-saturated then the existence of a ball of radius $\gamma$ with $\gamma<\gamma_0-k$ for all $k\in \mathbb{N}$ follows by saturation.
\end{proof}

\begin{lemma}\label{L:getting almost internal}
Let $S=\{S_t:t\in T\}$ be an infinite definable family of pairwise distinct %{\bf I changed disjoint to distinct, please check} 
finite sets  of $0$-balls. Assume that there exists an integer $n\in\mathbb{N}$ with $|S_t|\leq n$ for all $t\in T$.
Then there exists a definable infinite subset $T'\subseteq T$ and a definable finite-to-one $f$ from $T'$ into either $\Gamma$ or $K/\CO$.
\end{lemma}
\begin{proof}
We may assume that $K$ is $\aleph_0$-saturated, by passing to a large enough elementary extension. 
%Notice that we may prove the result in an elementary extension, thus we may assume that $\CM$ is $\omega$-saturated.

By possibly passing to a definable subset of $T$, we may assume that $|S_t|=|S_{t^\prime}|=n$ for all $t,t^\prime\in T$, for some $n\in \mathbb{N}$. 

Given two $0$-balls, $b_1\neq b_2$ denote $$d(b_1,b_2)=\{v(x_1-x_2): \mbox{ for any } x_i\in b_i\},$$ and $d(b_1,b_1)=1$ (notice that $d(b_1,b_2)$ is independent of the choice of $x_i\in b_i$). And for $t\in T$, let $d(t)=\min\{d(b_1,b_2):b_i\in S_t\}$ and note that  $d(t)\leq 1$.
%(the valuative diameter of $\bigcup S_t$).

%(I think that) we may assume that for every $t$, and every $b1\neq b2$ in $S_t$, $d(b1,d2)=d(t)$ (we just choose pairs which realize the distance

% Let $d(S)=min\{d(S_t):t\in T\}$ (exists since it is a definable subset of $\Gamma$).

\vspace{.2cm}

\noindent{\bf Case 1} There exists $d\in \Gamma$ such that $\{t\in T:d(t)=d\}$ is infinite.

\vspace{.2cm}

Reducing $T$, we may assume that for every $t\in T$, $d(t)=d$.
\vspace{.2cm}

{\bf Case 1(a)} $d\in \mathbb{Z}$.

\vspace{.2cm}

Given a $0$-ball $b$, and $\gamma<0$ let $B_{\gamma}(b)$ be the unique ball of radius $\gamma$ containing $b$.
With this notation, for $t\in T$ and $b,b'\in S_t$ we get  $B_{d-2}(b)=B_{d-2}(b')$ (since $d(b_1,b_2)=d$).  In other words, for $t\in T$ the ball $B_t:=B_{d-2}(b)$ for some $b\in S_t$ is independent of the choice of $b\in S_t$.

Define $t_1\sim t_2$ if $B_{t_1}=B_{t_2}$.
%for all (equivalently some) $b_1\in S_{t_1}$, and $b_2\in S_{t_2}$, $B_{d-2}(b_1)=B_{d-2}(b_2)$.
Because $d-2\in \mathbb Z$, each $B_t$ contains only finitely many $0$-balls (Lemma \ref{L:ball in elem exten of p-adics}). In fact each $B_t$ contains at most $r:=p^{d-2}$-many $0$-balls. %there is $p^{d-2}$ ({\bf something like $p^{d-2}?$}. 
Thus each $\sim$-class contains at most $r$-many elements.
 Therefore, the map $t\to B_t$ from $T$ into $K/p^{d-2}\CO$ has fibres of cardinality at most $r$. Since $K/r\CO$ is definably isomorphic to $K/\CO$, we are done.
\vspace{.2cm}

{\bf Case 1(b)} $d+k<0$ for all $k\in \mathbb{N}$.

\vspace{.2cm}

For each $t\in T$, let $S'_t:=\{b\in S_t:(\exists b'\in S_t) \,  d(b,b')=d\}$. Replacing $S_t$ with $S'_t$ we may assume that every $b\in S_t$
has some $b'\in S_t$ such that $d(b,b')=d$.

Let $\mathcal{B}$ be the collection of all balls of radius $d+1$ in $K$. For any $B\in \mathcal{B}$ there is no $t\in T$ with $\bigcup S_t\subseteq B$, for otherwise we could find $b,b'\in S_t$ with $d(b,b')=d$ and $b,b'\subseteq B$, which is absurd.
%But then there are some $a\in b$ and $a'\in b'$, $v(a-a')= d$ contradicting that $v(a-a')>d+1$.
Let $X=\bigcup_{t\in T} S_t$.  By assumption, $X$ is a union of infinitely many $0$-balls. By Lemma \ref{L:dugald} $X$ contains some ball $B_\gamma (x_0)$ with $\gamma+k<0$ for all $k\in \mathbb{N}$.  By Lemma \ref{L:ball in elem exten of p-adics}(2) $B_\gamma (x_0)$ intersects infinitely many of the $S_t$.

Consider the ball $B_0:=B_{d+1}(x_0)\in \mathcal{B}$. If $d+1\leq \gamma$ then $B_0\supseteq B_\gamma (x_0)$ and hence it clearly intersects infinitely many of the $S_t$. If $d+1>\gamma$ then $B_0\sub B_\gamma (x_0)$, thus (since $d<\mathbb Z$) $B_0$ contains infinitely many $0$-balls inside $B_{\gamma}(x_0)$. Because every such ball belongs to some $S_t$ it follows that $B_0$
intersects infinitely many of the $S_t$ as well.

In abuse of notation we write $B_0\cap S_t$ for the set of balls in $S_t$ which are contained in $B_0$.
 If $|B_0\cap S_t|\leq 1$ for all $t\in T$, then by passing to a definable subset of $T$ we may assume that $|B_0\cap S_t|=1$ for all $t\in T$ and as a conclusion we get a definable function (with finite fibers) from $T$ into $K/\CO$. Otherwise, by replacing $S$ with $\{S_t\cap B_0: t\in T\}$, we may finish the proof by induction on $K$,  since as we noted above,  $|S_t\cap B_0|<n$.

\vspace{.2cm}

\noindent{\bf Case 2} For every $d\in \Gamma$ the set $\{t\in T:d(t)=d\}$ is finite.

\vspace{.2cm}
In this case the map $t\mapsto d(t)$  a finite-to-one function into $\Gamma$.
\end{proof}

We are now in position to obtain the second goal of our strategy:

\begin{proposition}\label{Internal}
        Let $X=\{X_t:t\in T\}$ be a definable family of  infinite pairwise disjoint subsets of $K$. 
        Then, 
        
        \begin{enumerate}
        \item There exists a definable set $D\sub K$ such that for each $t\in T$, the set $D\cap X_t$ is the  union of finitely many balls of equal radius.
            
            \item There exists a definable infinite subfamily $T^\prime\subseteq T$ and a definable finite-to-one $f$ from $T^\prime$ to either $\Gamma$ or $K/\CO$.
                \end{enumerate}
                
\end{proposition}
\begin{proof}
    Recall that a definable set $D\subseteq K$ is bounded if it is contained in some ball $B_\gamma(0)$.  We first reduce to the case where all the $X_t$ are bounded, by replacing each $X_t$ with the set
     $B_{\gamma(t)}(0)\cap X_t,$ for $\gamma(t)\in \Gamma$ maximal satisfying $\mathrm{Int}(B_{\gamma}(0)\cap X_t)\neq \emptyset$ (since each $X_t$ is infinite and $K$ has no definable infinite discrete sets, each $t$ has some $\gamma\in \Gamma$ such that $B_\gamma(0)\cap X_t$ is infinite, so has non empty interior). The family $\{B_{\gamma(t)}(0)\cap X_t:t\in T\}$ is definable, so we may assume that each $X_t$ is bounded.

    Consequently, for any $t\in T$ the set $\{\gamma\in \Gamma: (\exists a\in X_t)(B_\gamma(a)\subseteq X_t)\}$ has a minimum. Let $\gamma_t\in \Gamma$ be this minimum. By minimality of $\gamma_t$ and Lemma \ref{L:dugald}, for any $t\in T$ $X_t$ contains only finitely many balls of radius $\gamma_t$. We may now choose, uniformly in $t$, the set $D_t\sub X_t$, consisting of the union of those balls of radius $\gamma_t$ which are contained in $X_t$. Let $D=\bigcup_t D_t$. This ends (1).

    Consider the definable map $T\to \Gamma$ sending $t$ to $\gamma_t$. 
    
    If the  map is finite-to-one, we are done.
Otherwise, there exists $\gamma\in \Gamma$ with $\{t\in T: \gamma_t=\gamma\}$ infinite. By passing to a definable subfamily, we may assume that $\gamma_t=\gamma$ for all $t\in T$. After rescailing, we may assume that $\gamma=0$. We are now in the situation to apply Lemma \ref{L:getting almost internal}.
\end{proof}

\subsection{Interpretable fields}
We now return to our problem of an interpretable field $\CF$.
We start with the  following general lemma. It uses the well known coding of finite sets using symmetric functions:

\begin{fact}\label{F:EfI}
Let $(L,+,\cdot,\dots)$ be any field, possibly with  additional structure. Then  $L$ eliminates bounded finite imaginaries. I.e. if $\{X_t\subseteq L^n:t\in T\}$ is a definable family of finite sets uniformly bounded in size then there exists a definable map $f:T\to L^m$, for some integer $m$, satisfying that $f(t_1)=f(t_2)$ if and only if $X_{t_1}=X_{t_2}$.
\end{fact}

\begin{lemma}\label{1-dim}
    Let $\CF$ be a field interpretable in an $\aleph_0$-saturated structure $\CM$. Then there exists an infinite definable subset $S\subseteq \CF$ that is in definable bijection with $M/E$  for some definable equivalence relation $E$ on $M$.
\end{lemma}

\begin{proof}
    For any infinite definable subset $S\subseteq \CF$ there exist an integer $n$, definable subset $X\subseteq M^n$ and a definable equivalence relation $E$ such that $S$ is in definable bijection with $X/E$.

    Choose $S$ so that $n$ is minimal possible and let $X$ and $E$ be the corresponding definable set $X\subseteq K^n$ and definable equivalence relation, respectively. Assume, for simplicity, that $S=X/E$. We claim that $n=1$. Otherwise, let $\pi: X\to M^{n-1}$ be the projection onto the first $n-1$ coordinates. Let $W=\pi(X)$. If for some $w\in W$ the set $X_w/E$ is infinite, where $X_w:=X\cap \pi^{-1}(w)$ the set $X_w$ contradicts the minimality of $n$ (since we can definably identify $X_w$ with a subset of $M$).

So  $|X_w/E|$ is finite for all $w\in W$. By saturation, there exists a uniform bound $k$ on $|X_w/E|$ as $w$ ranges over $W$. Reducing $X$, we may assume that $|X_w/E|=r$ (for some $r\le k$) for all $w\in W$. This gives a definable correspondence $C:W\to \CF$ given by $w\mapsto X_w/E$. Using Fact \ref{F:EfI} we can replace $C$ with a function $c:W\to \CF^r$ whose image is infinite. So there is a projection $\tau: \CF^r\to \CF$ such that $\tau\circ c$ has infinite image in $\CF$. Define an equivalence relation $E'$ on $W$ by $E'(u,v)$ if $\tau\circ c(u)=\tau\circ c(v)$. This gives us a bijection between an infinite subset of $\CF$ and $W/E'$. Since $W\subseteq M^{n-1}$, this contradicts the minimality of $n$.

Hence $S$ is in definable bijection with $X/E$, where $X\subseteq M$ is a definable subset. By possibly enlarging $S$ by one element, we may assume that $X=M$.
\end{proof}

We now add our underlying assumption that $K$ is a $P$-minimal field.
We first note that neither $K/\CO$ nor $\Gamma$ ``eliminate $\exists^\infty$'' and this remains true for any infinite definable subset.

\begin{lemma}\label{L:no universal finiteness}
    Let $V$ be either $\Gamma$ or $K/\CO$ with the induced structure. Let $X$ be an infinite definable subset of $V$. Then there exists a formula $\varphi(x,y)$ satisfying that for every $n<\omega$ there exists $b_n$ such that $\varphi(V,b_n)\subseteq X$,  $|\varphi(V,b_n)|<\omega$ and $\sup_n |\varphi(V,b_n)|=\omega$. 
\end{lemma}
\begin{proof}
    As usual, there is no harm assuming that $K$ is $\aleph_0$-saturated. 
    Assume that $V=\Gamma$. Since $X$ is infinite, by quantifier elimination, we may assume without loss of generality that $X$ is of the form  $\{x\in \Gamma:\gamma_1<x<\gamma_2 \land P_m(x)\}$ for some $\gamma_1,\gamma_2\in \Gamma\cup\{\pm \infty\}$. Fix some $a\in X$. Since $X$ is infinite, at least one of $\gamma_1$ or $\gamma_2$ is not in the same archimedean component as $a$. So the set $\phi(x,n):=x\in X\cap (a-n,a+n)$ satisfies the requirements.

    Assume that $V=K/\CO$. Let $\pi: K\to K/\CO$ be the natural projection and $Y=\pi^{-1}(X)$. Since $X$ is infinite, $Y$ contains infinitely many balls of radius $0$. By Lemma \ref{L:dugald}, there exists $a\in Y$ such that $B_{-k}(a)\subseteq Y$ for all $k\in \mathbb{N}$. By Lemma \ref{L:ball in elem exten of p-adics}, for each $k\in \mathbb{N}$, $B_{-k}(a)$ contains only finitely many, say $n_k$, balls of radius $0$. In fact, $\sup_{k\in \mathbb{N}} n_k=\omega$ by Lemma \ref{L:ball in elem exten of p-adics}(2). By choosing elements $b_k\in K$ with $v(b_k)=-k$, the definable sets $\pi(B_{v(b_k)}(a))\subseteq X$ satisfy the requirements.
 \end{proof}

\begin{proposition}\label{nointernal}
    %Let $K$ be a p-adically closed field and $\CF$ an interpretable field.
    For $V=\Gamma$ or $V=K/\CO$, there is no infinite definable subset $I\subseteq \CF$ and finite-to-one definable function $f:I\to V$.
\end{proposition}
\begin{proof}
    Assume towards a contradiction that such $I\subseteq \CF$ and $f:I\to V$ existed. Since $I$ is infinite and $f$ is finite-to-one, $f(I)\subseteq V$ is a definable infinite subset. Let $\varphi(x,y)$ and $\{b_n:n\in \mathbb{N}\}$ be as provided by Lemma \ref{L:no universal finiteness}. Let $\psi(u,y)$ be the formula $(\exists x)(u\in I\wedge f(u)=x\wedge \varphi(x,y))$. Since $f$ is finite-to-one, for every $n\in \mathbb{N}$, $n<|\psi(\mathcal{F},b_n)|<\omega$. On the other hand, $\mathcal{F}$, with its induced structure, is a field of finite $dp$-rank, and thus we get a contradiction to \cite[Corollary 2.2]{DoGoStrong}.
\end{proof}

We are ready to show that no interpretable field contains an infinite definable subset almost internal to either $\Gamma$ or $K/\CO$:

\begin{corollary}\label{C:existence of I}
    %Let $K$ be a $p$-adically closed field, $\CF$ an infinite field interpretable in $K$.
    There exists a definable dp-minimal $I\subseteq \CF$, and a definable injection $I\hookrightarrow K$.
\end{corollary}
\begin{proof}
    Let $X\subseteq K^n$ be a definable set $E$ a definable equivalence relation on $X$ such that $X/E$ is the universe of $\CF$. By Lemma \ref{1-dim} there exists a definable infinite $S\sub \CF$ and an equivalence relation $E'$ on $K$ such that $K/E'$ is in definable bijection with $S$. 
    
    We first claim that only finitely many $E'$-classes could be infinite. Indeed,
    if $E'$ had infinitely many infinite classes then by restricting to those classes with nonempty interior,  we may assume that all classes are infinite. By applying Proposition \ref{Internal}, we obtain a definable injection of some infinite subset of $\CF$ into either $\Gamma$ or $K/\CO$,  contradicting Proposition \ref{nointernal}. 
    
    We may, therefore, assume that all $E'$-classes are finite, and by uniform finiteness, can restrict to the case where all classes are  of  size $m$ for some $m\in \mathbb N$.
    The definable quotient map $\pi:K\to K/E'$ gives rise to  a definable family $\{\pi^{-1}(e):e\in K/E'\}$ of distinct finite sets of cardinality $m$. Hence by Fact \ref{F:EfI}, there exists a definable injection $\widehat f:K/E'\to K^n$ for some integer $n$. Composing, we get a definable injection $f:S\to K^n$.
        %Define a map $S/E'\to K^m$ (where $S$ the union of all $E'$ classes), as follows: for a set $A:=\{a_1,\dots, a_m\}\sub K$ set $\mathrm{Sym}(A)$ the image of $A$ under all symmetric functions in $m$ variable (in their natural order), and set $[e]\mapsto \mathrm{Sym}([e])$ for all equivalence class $[e']$ in $E'$. Since $\mathrm{Sym}$ is an injection on subsets of $K$ of size $m$ we get a definable injection from  $I$ into an (infinite) subset of $K^m$.

        Since $\dpr(S)=1$ the set $f(S)\sub K^n$ is one-dimensional and hence, by \cite[Proposition 4.6]{SimWal}, there exists a definable infinite $V\subseteq f(S)$ and a projection map $\pi:K^n\to K$ such that $\pi|V$ is a bijection with an open subset of $K$.
\end{proof}

\section{Subset of $\CF$ strongly internal to $K$}
Let $K$ be a $P$-minimal valued. The previous section was concluded with the proof that there exists an infinite $I\subseteq \mathcal{F}$ ``strongly internal'' to $K$, where we borrow the following terminology from \cite{HaPetRCVF}:
\begin{definition}
    A definable set $S\subseteq \CF$ is {\em strongly internal to $K$ over $A$} if there exists an $A$-definable injection $f:S\to K^n$ for some $n$. It is called {\em strongly internal to $K$} if it is strongly internal over some $A$.
\end{definition}

Since, by \cite[Theorem 0.3.$\oplus_0$]{Simdpr}, dp-rank is additive we may conclude:
\begin{remark}\label{R:additivity for dpr in strongly}
If $S$ is strongly internal to $K$ over $A$ then for any $a,b\in S$
\[\dpr(a,b/A)=\dpr(a/bA)+\dpr(b/A).\]
\end{remark}
%In this section we assume that $K$ is a dp-minimal expansion of a non-archimedean valued field in which $\acl$ satisfies exchange, and infinite definable subsets of $K$ have nonempty interior. Let $\CF$ an infinite interpretable field. 

In the present section we study subsets of $\CF$ strongly internal to $K$ of maximal dp-rank. Our aim is to show that such sets, at least on some generic subset, are not too far from being closed under the field operations. The proof is built on the analogous statement from \cite{HaPetRCVF}. Specifically, if $Y\sub \CF$ is of maximal dp-rank among all the subsets of $\CF$ strongly internal to $K$.  We show, Lemma \ref{L:generic linearity for K}, that the function $\la x,y,z\ra \mapsto xy-z$ maps a generic subset of $(I+b) \times Y^2 $  (some $b$) into $Y$.

%{\bf I prefer to use $\dim$ when working in $K$, but I left it as $\dpr$ for now.}
\begin{lemma}\label{L:locally the same}
    Let $X\subseteq Z\subseteq K^n$ be $\emptyset$-definable sets with $\dpr(X)=\dpr(Z)$. For any $d\in X$ with $\dpr(d)=\dpr(X)$ there exists an open neighborhood $U\subseteq K^n$ such that  $U\cap X=U\cap Z$.
\end{lemma}
\begin{proof}
    The relative interior of $X$ in $Z$ is the set $D$ of $x\in X$ such that there exists $U\ni x$ an open subset of $K^n$ such that $U\cap Z\subseteq X$. By \cite[Corollary 4.4]{SimWal}, $\dpr(X\setminus D)<\dpr(X)$ and hence $d\in D$ so there is an open set $U\ni d$ such that $U\cap Z=U\cap X$.
\end{proof}

The following lemma plays an important role in our argument.
\begin{lemma}\label{L:generic smaller ball}
Let $V\subseteq K^n$ be an open set,   $x\in K^m$  any element and $A$  an arbitrary set of parameters. Then there exists $B\supseteq A$ and  a $B$-definable open subset $U\subseteq V$ such that  $\dpr(x/B)=\dpr(x/A)$.
Moreover, if $x\in V$ then we can find such  $U$ with  $x\in U$.
\end{lemma}
\begin{proof}
    Since $V$ is open, and is not assumed to be defined over $A$, we may assume that  $V=B_{v(g_1)}(a_1)\times \dots \times B_{v(g_{n})}(a_{n})$ for some $a_1,\ldots, a_n\in K$ and $g_1,\dots, g_{n}\in \Gamma$.

    We prove the result by induction on $n$. For $n=1$, we use compactness to find $r_1\in K$ with $v(r_1)>v(g_1)$ such that
    $\dpr(r_1/xA)=1$ and then find $e_1\in B_{v(r_1)}(a_1)$, such that $\dpr(e_1/r_1xA)=1$. We have $U:=B_{v(r_1)}(e_1)\sub B_{\gamma_1}(a_1)$ and by exchange, $\dpr(x/e_1r_1A)=\dpr(x/A)$, so $B=e_1r_1A$ satisfies the lemma.  If $x\in V$ then we may start with $a_1=x$, and then  $x\in B_{v(r_1)}(e_1)$.

    As for the general case, we first replace, by induction, $B_{v(g_1)}(a_1)\times \dots \times B_{v(g_{n-1})}(a_{n-1})$ by an open subset $U_1\sub K^{n-1}$ definable over $A_1\supseteq A$ such that $\dpr(x/A_1)=\dpr(x/A)$. Then we apply the case $n=1$ to the last coordinate with $A_1$ replacing $A$.
   % Assuming we proved the result for $n-1$, we replace $A$ by $
   % For any $i\leq n$, since $(v(g_i),\infty)$ is infinite, we can find, by saturation and using Exchange, inductively non zero %elements $r_i,e_i\in K$ with $v(r_i)>v(g_i)$ and
   % $\dim(r_i/xAr_{<i}e_{<i})=n$,  and $e_i\in B_{v(r_i)}(a_i)$ with $\dim(e_i/xAr_{\leq i}e_{<i})=n$ ???????. We take %$U=B_{v(r_0)}(e_0)\times\dots\times B_{v(r_{n-1})}(e_{n-1})$ and $B=Ae_{<n}r_{<n}$. Since $K$ has exchange, the result follows.
    %As for the moreover, if $x\in V$ then we just replace $(a_0,\dots,a_{n-1})$ above by $x$.
\end{proof}

%\begin{lemma}\label{L:generic smaller ball}
%Let $V\subseteq K^n$ be an open set, let $a\in V$ and $A$ be any set. Then there exists a $B$-definable open subset $a\in U\subseteq V$ for $A\subseteq B$ with $\dpr(a/B)=\dpr(a/A)$.
%\end{lemma}
%\begin{proof}
 %   Assume that $a=(a_0,\dots,a_{n-1})$ and assume that $B_{v(g_0)}(a_0)\times \dots \times B_{v(g_{n-1})}(a_{n-1})\subseteq V$ for some $g_0,\dots, g_{n-1}\in \Gamma$. For any $i<n$, since $(v(g_i),\infty)$ is infinite,  by saturation we can find inductively non zero elements $r_i,e_i\in K$ with $v(r_i)>v(g_i)$ and $r_i\dclindep Ar_{<i}e_{<i}$ and $e_i\in B_{v(r_i)}(a_i)$ with $e_i\dclindep aAr_{\leq i}e_{<i}$, where $\dclindep$ denotes $\acl$-independence. We take $U=B_{v(r_0)}(e_0)\times\dots\times B_{v(r_{n-1})}(e_{n-1})$ and $B=Ae_{<n}r_{<n}$. Since $K$ has exchnge, the result follows.
%\end{proof}

We can now prove the main lemma of this section that is the main technical lemma of the paper:

\begin{lemma}\label{L:generic linearity for K}
Let $I\sub \CF$ be as provided by Corollary \ref{C:existence of I} and $Y\sub \CF$ strongly internal to $K$ of maximal dp-rank. Assume that both are defined over a parameter set $A\sub K$. 

Then, there exists a definable $J\subseteq I$ and a definable  $S\subseteq Y^2$ such that $\dpr(J)=1$ and  $\dpr(S)=2\dpr(Y)$, and there exists $b\in J$, such that for every $\la x,y,z\ra\in J\times S\sub J\times Y^2$, we have $(x-b)y+z\in Y$.

Furthermore, if $\la b,c,d\ra \in I\times Y^2$ are such that $\dpr(b,c,d/A)=2\dpr(Y)+1$ then we can choose $J$ and $S\ni \la c,d\ra$, definable over some set $B\ni b$ such that $\dpr(c,d/B)=2\dpr(Y)$.
	\end{lemma}
	\begin{proof} We fix $A$-definable injections $g_1:I\to K$ and $g_2:Y\to K^l$.
	For simplicity of notation assume $A=\emptyset$. By assumption $\dpr(I)=1$, and denote $\dpr(Y)=n$
	for some integer $n>0$ and fix some $\la b,c,d\ra \in I\times Y\times Y$ such that $\dpr(b,c,d)=2n+1$.
	Note that this implies that
	\[
	    2n+1=\dpr(b,c,d)\leq \dpr(b,d/c)+\dpr(c)\leq \dpr(b,d)+\dpr(c)\leq 2n+1
	\]
	and hence $\dpr(b,d/c)=n+1$, $\dpr(c)=n$. Similarly, $\dpr(b,c)=n+1$ and $\dpr(b/c)=1$. 
	
	For $\la x,y, z\ra \in I\times Y\times Y$ consider the function $f_y(x,z)=xy-z$. Let $e=f_c(b,d)$.

\begin{claim}
 $b\notin \acl(c,e)$.
\end{claim}
\begin{claimproof}
Assume towards a contradiction that there existed an algebraic formula $\varphi(x,e,c)$ isolating $\tp(b/c,e)$, in particular, $\varphi(x,e,c)$ would imply that $x\in I$. By the definition of $f_c$, $d\in \dcl(b,c,e)$, therefore  $\phi(x,e,c)\wedge f_c(x,y)=e$ is an algebraic formula isolating $\tp(b,d/c,e)$.  Hence, there is some integer $m$ with $\exists^{=m}(x,y)(\phi(x,e,c)\wedge f_c(x,y)=e)$. By compactness, there is a formula $\psi(z,c)\in \tp(e/c)$ that implies $\exists^{=m}(x,y)(\phi(x,z,c)\wedge f_c(x,y)=z)$. In other words, for $D:=\psi(\CF,c)$, there is a 1-to-$m$ definable partial $c$-definable correspondence $G$ from $D$ into $I\times Y$, sending $e$ to $(b,d)$.
Note that if $e_1\neq e_2\in D$ then $f_c^{-1}(e_1)\cap f_c^{-1}(e_2)=\0$, thus $G(e_1)\cap G(e_2)=\emptyset$.

The image of $G$ in $I\times Y$ is a $c$-definable set containing $\la b,d\ra$ and since $\dpr(b,d/c)=n+1$, we have 
 $\dpr(\mathrm{Im}(G))=n+1$.

Recall that $(g_1,g_2):I\times Y\to K\times K^l$ is a $\0$-definable injection, so $H=(g_1,g_2)\circ G:\mathrm{Dom}(G)\to K^{1+l}$  is a one-to-$m$ correspondence,
with disjoint images corresponding to distinct $e_1,e_2\in \mathrm{Dom}(G)$.  By Fact \ref{F:EfI}, $H$ induces a definable injection from $\mathrm{Dom}(G)$ into $K^N$, for some integer $N$, as $\dpr(\mathrm{Im}(G))>\dpr(Y)$ this is a contradiction to the choice of $Y$.
\end{claimproof}

We thus conclude that $b\notin \acl(c,f_c(b,d))=\acl(c,e)$ and in particular $f_c^{-1}(e)\subseteq I\times Y$ is infinite.
Notice that by the definition of $f_c$, for every $b'\in I$, and $d_1\neq d_2\in Y$, we have $f_c(b_1,e_1)\neq f_c(b_1,e_2)$,
thus the projection of $f_c^{-1}(e)$ on the first coordinate, call it $J$, is infinite.
By definition of $J$, for every $x\in J$ there is $z\in Y$ with $xc-z=bc-d=e$ and since $\mathcal{F}$ is a field the map from $J$ to $Y$ mapping $x\in J$ to $z=xc-e=(x-b)c+d$ is injective.

Because $g_1$ is injective, $g_1(J)$ is an infinite subset of $K$ and hence has non-empty interior.
%Recall that in the valued field sort of p-adically closed fields $\acl$ satisfies the exchange property and hence it induces a pregeomtry. Furthermore, the pregeometric dimension is equal to the dp-rank. We write $\dclindep$ for the independence relation on the home sort induced by this pregeometry.
%Since $f_1(J)$ is infinite it has non-empty interior, i.e. there exists a ball $B_{v(g)}(a)\subseteq f_1(J)$ for some $g,a\in K$.
By Lemma \ref{L:generic smaller ball}, we can find a $B$-definable open subset $U\subseteq g_1(J)$, such that 
\[
\dpr(g_2(c),g_2(d)/Bg_1(b))=\dpr(g_2(c),g_2(d)/g_1(b))=2n.
\]
Set $J^\prime =g_1^{-1}(U)$. It is an infinite $B$-definable subset of $J$, so for every $x$ in $J'$, we have $(x-b)c+d\in Y$. Let
\[
S=\{\la y,z\ra\in Y^2: (\forall x\in J')((x-b)y+z\in Y)\}.
\]
Note that $\la c,d\ra\in S$ and that $S$ is $Bb$-definable.  Because $(c,d,b)$ and  $(g_1(b),g_2(c),g_2(d))$ are inter-definable over $\0$
it follows that $\dpr(c,d/Bb)=2n$, so the proof is completed.
\end{proof}

\begin{remark}
Note that in Lemma \ref{L:generic linearity for K} we do not claim that the set  $I\times Y\times Y$ is mapped under $xy+z$, (or under $(x-b)y+z$) onto a subset of $\CF$ whose dp-rank equals $\dpr(Y)$ (this will turn out to be true once we complete the proof of the main theorem). Instead, at this stage, we  only found {\em a subset}
$J\times S\sub I\times Y\times Y$ of full dp-rank whose image has the same rank as $\dpr(Y)$.
\end{remark}

\section{infinitesimal neighbourhoods and topology}
%Throughout this section we use without further mention the fact that dp-minimal expansions of valued fields are tame uniform structures in the sense of \cite{SimWal}. 
In the present section we use methods similar to those in \cite{HaPetRCVF} in order to construct a type definable, ``infinitesimal'' subgroup of $(\CF,+)$ which is definably embedded  into $K^n$ for some $n$. The field $\CF$ itself will later be embedded into some $K^m$, using the subgroup of infinitesimals.\\

%{\bf We assume in this section that $K$ is $P$-minimal}. 
We assume in this section that $K$ is $P$-minimal. We will repeatedly use the fact (already mentioned in the intoduction) that in $P$-minimal $\acl$ satisfies the Steiniz Exchange prilnciple.
Throughout  $+$, $-$, $\cdot$ and $(\, )^{-1}$ denote the operations in  $\CF$.

% {\bf I feel that the comment below is vague, I prefer to drop it until we have something concrete}.
% \begin{remark}\label{R:definable}
%     All the results of the present and the next sections go through verbatim if $\CF$ is definable in $K$, but under the weaker requirement that $K$ is a dp-minimal field satisfying the assumptions of \cite{SimWal}. Specifically, these assumptions require that the field $K$ supports a definable uniform topology (that need not be a field topology) with no isolated points, and that every definable set of maximal dp-rank has non-empty interior. The only cross-reference to earlier sections is an application of Lemma \ref{L:generic linearity for K} which is automatic in case $\CF$ is definable.
% \end{remark}

%For the purpose of this section, let $K\prec \CK$ be a $|K|^+$-saturated elementary extension. However, unlike the usual convention, all definable sets will be realized in $K$ unless stated otherwise.

\begin{definition}
For any $Z\subseteq \CF$ and a definable injective  $g:Z\to K^m$, let
$\tau_{Z,g}$ be the topology on $Z$ given by $\{g^{-1}(U): U\subseteq K^m \text{ open}\}$
\end{definition}

We observe that
because $K^n$ has a definable basis for its topology (given, say, by the family of open balls), each $\tau_{Z,g}$ has a definable basis as well. Also, if $Z_1,Z_2$ are both strongly internal to $K$ via definable injections $g_1$ and $g_2$, respectively,  then $\tau_{Z_1\times Z_2,g_1\times g_2}$ is equal to the topology generated by $\tau_{Z_1,g_1}\times \tau_{Z_2,g_2}$.

We will repeatedly use the following

\begin{lemma}\label{F:cont-points}
    Let $Z_1,Z_2, \subseteq \CF$ be strongly internal to $K$, witnessed by  $A$-definable injections $g:Z_1\to K^m$ and $h:Z_2\to K^n$. Let $f:Z_1\to Z_2$ be an $A$-definable function between the topological spaces $(Z_1,\tau_{Z_1,g})$ and $(Z_2,\tau_{Z_2,h})$.
    Then the set $C\sub Z_1$ of $\tau_{Z_1,g}$-continuous points of $f$ is $A$-definable, and $\dpr(Z_1\setminus C)<\dpr(Z_1)$.

\end{lemma}
\begin{proof}The definability of $C$ follows from the definability of a basis for $\tau_{Z,g}$. The dimension statement follows from the analogous result for definable functions on subsets of $K^m$, in $P$-minimal fields, \cite[Theorem 5.1]{HasMac}.
\end{proof}

We thus have:
\begin{lemma}\label{unique topology}Let  $Z\sub \CF$ be definable and  $g:Z\to K^m$, $h:Z\to K^n$ two $A$-definable injections. Then,
$\tau_{Z,g}$ and $\tau_{Z,h}$ agree at every $z\in Z$ with $\dpr(z/A)=\dpr(Z)$.
Namely, there is a common basis for the $\tau_{Z,g}$-neighbourhoods and the $\tau_{Z,h}$-neighbourhoods of $z\in Z$.
\end{lemma}
\begin{proof} Apply Lemma \ref{F:cont-points} to $\id:Z\to Z$.\end{proof}

\begin{definition} For $Z\sub \CF$ a definable set, $g:Z\to K^m$ a definable injection, and $d\in Z$,
 let $\nu_{Z,g}(d)$ be the partial type given by all definable $\tau_{Z,g}$-open sets containing $d$.
We call it the {\em infinitesimal neighborhood of $d$ with respect to $\tau_{Z,g}$}.
\end{definition}

\begin{remark}\label{R:product of nu}
If $Z_1, Z_2$ are both strongly internal to $K$ over $A$ via injections $g_1, g_2$, respectively, and $\la d_1,d_2\ra \in Z_1\times Z_2$ with $\dpr(d_1,d_2/A)=\dpr(Z_1)+\dpr(Z_2)$ then $\nu_{Z_1,g_1}(d_1)\times \nu_{Z_2,g_2}(d_2)=\nu_{Z_1\times Z_2,g_1\times g_2}(d_1,d_2)$ (see above discussion on $\tau_{Z_1,g_1}\times \tau_{Z_2,g_2}$).
\end{remark}

By Lemma \ref{unique topology}, we have:
\begin{corollary} If $Z\sub \CF$ is strongly internal to $K$ over $A$ and $d\in Z$ is such that $\dpr(d/A)=\dpr(Z)$ then the infinitesimal neighborhood of $d$ in $Z$ does not depend on any particular $A$-definable injection of $Z$ into some $K^m$. We denote it by $\nu_Z(d)$.
\end{corollary}

We also have:
\begin{lemma}\label{internal-locally the same}Assume that $Z\sub \CF$ is strongly internal to $K$ over $A$ as witnessed by $g$, and $Z_1\sub Z$ is $A$-definable with $\dpr(Z_1)=\dpr(Z)$.
If $d\in Z_1$ is such that $\dpr(d/A)=\dpr(Z)$ then $\nu_{Z_1}(d)=\nu_{Z}(d)$.
\end{lemma}
\begin{proof} By Lemma \ref{L:locally the same}, the topologies $\tau_{Z,g}$ and $\tau_{Z_1,g}$ agree on a neighborhood of $d$, thus $\nu_Z(d)=\nu_{Z_1}(d)$.\end{proof}

\begin{lemma} \label{function}
%\begin{enuemrate}
%\item
Let $Y_1,Y_2\sub \CF$ be strongly internal to $K$ over $A$. If $f:Y_1\to Y_2$ is an $A$-definable partial function,
and $a\in \dom(f)$ with $\dpr(a/A)=\dpr(Y_1)$ then $f$ takes the partial type $\nu_{Y_1}(a)$ into $\nu_{Y_2}(f(a))$. If $f$ is injective then $f(\nu_{Y_1}(a))=\nu_{Y_2}(f(a))$.

%\item Assume that $Y_1,Y_2,Y_3\sub \CF$ are $A$-definable and strongly internal to $K$,  with $\dpr(Y_1)=\dpr(Y_2)=\dpr(Y_3)=n$, and %assume $G\sub Y_1\times Y_2\times Y_3$, such that for every $(y_1,y_2)\in Y_1\times Y_2$ there is at most one $y_3$ such that %$(y_1,y_2,y_3)\in G$ and for every $(y_1,y_3)\in Y_1\times Y_3$ there is at most one $y_2\in Y_2$ such that $(y_1,y_2,y_3)\in G$.

%   Let $(a_1,a_2,a_3)\in G$ with $\dpr(a_1,a_2,a_3/A)=2n$.
%    We now work in a $|K|^+$-saturated extension of $K$. Then for every $\al_1\models \nu_{Y_1}(a_1)$, $G(\al_1,y_2,y_3)$ defines a %bijection of $\nu_{Y_2}(a_2)$ and $\nu_{Y_3}(a_3)$.
 %   \end{enumerate}
    \end{lemma}
    \begin{proof} By Lemma \ref{F:cont-points}, $f$ is continuous at $a$ with respect to $\tau_{Y_1,g}$ and $\tau_{Y_2,h}$ (for any $A$-definable $g,h$ witnessing the strong internality of $Y_1$, $Y_2$, respectively). It is now easy to conclude that $f$ maps $\nu_{Y_1}(a_1)$ into $\nu_{Y_2}(f(a))$.\end{proof}

The main result in this section is the following.
\begin{lemma} \label{addition closure} Assume that $Y_1,Y_2,Y_3\sub \CF$ are strongly internal to $K$ over $A$, with $\dpr(Y_1)=\dpr(Y_2)=\dpr(Y_3)=k$, and assume that $Y_1+Y_2\sub Y_3$.
Then
\begin{enumerate}
\item For every $c\in Y_1$ and $d\in Y_2$ such that $\dpr(c,d/A)=2k$, we have $\nu_{Y_1}(c)-c=\nu_{Y_2}(d)-d$, and for every $d'\in Y_2$ with $\dpr(d'/A)=k$ we have  $\nu_{Y_2}(d)-d=\nu_{Y_2}(d')-d'$.
    \item For every $d\in Y_2$ such that $\dpr(d/A)=k$, the partial type $\nu_{Y_2}(d)-d$ is a type definable subgroup of $(\CF,+)$
    \end{enumerate}
\end{lemma}
\begin{proof}
It is convenient to work in a $|K|^+$-saturated elementary extension $\CK$ of $K$ and work with the realizations in $\CK$ of the infinitesimal types. We first make some general observations.

 Consider the function $F:Y_1\times Y_2\to Y_3$, $F(y_1,y_2)=y_1+y_2$.
Applying Lemma \ref{function} to $F$ (and using Remark \ref{R:product of nu}), we see that $F(\nu_{Y_1}(c)\times \nu_{Y_2}(d))\sub \nu_{Y_3}(c+d)$.
Consider also the function $H:Y_1\times Y_3\to Y_2$, $H(y_1,y_3)=y_3-y_1$. By Lemma \ref{function} (and Remark \ref{R:product of nu}), it sends $\nu_{Y_1}(c)\times \nu_{Y_3}(c+d)$ into $\nu_{Y_2}(d)$.
It follows that for every $c_1\in \nu_{Y_1}(c)$, the function $y\mapsto c_1+y$ is a bijection between $\nu_{Y_2}(d)$ and $\nu_{Y_3}(c+d)$. Indeed, $F(c_1,-)$ is a function from $\nu_{Y_2}(d)$ into $\nu_{Y_3}(c+d)$, whose inverse is $H(c_1,-)$.

By the same observations we show that for every $d_1\in \nu_{Y_2}(d)$, the function $x\mapsto x+d_1$ is a bijection of $\nu_{Y_1}(c)$ and $\nu_{Y_3}(c+d)$.

\vspace{.2cm}

 In order to prove (1) we show that the sets of realizations in $\CK$ of the partial types $\nu_{Y_1}(c)-c$ and $\nu_{Y_2}(d)-d$ are equal.

 Let $c_1\in \nu_{Y_1}(c)$. By our above discussion, there is $d_1\in \nu_{Y_2}(d)$ such that $c_1+d_1=c+d$. Thus, $c_1-c\in \nu_{Y_2}(d)-d$, so $\nu_{Y_1}(c)-c\sub \nu_{Y_2}(d)-d$.
   The other inclusion is proved similarly, hence $\nu_{Y_1}(c)-d=\nu_{Y_2}(d)-d$.

Assume next that $d'\in Y_2$ is such that $\dpr(d'/A)=k$. We fix $c\in Y_1$ such that $\dpr(c/d,d'A)=k$.
It follows, by Remark \ref{R:additivity for dpr in strongly}, that $\dpr(c,d/A)=2k$ and $\dpr(c,d'/A)=2k$.
And then, by what we just saw,  $\nu_{Y_2}(d)-d=\nu_{Y_1}(c)-c=\nu_{Y_2}(d')-d'$.

\vspace{.2cm}

(2)
We claim that  $\nu_{Y_2}(d)-d$ is a subgroup of $(\CF,+)$. Given $d_1,d_2\in \nu_{Y_2}(d)$, we need to show that $(d_1-d)-(d_2-d)=d_1-d_2$ is also in $\nu_{Y_2}(d)-d$.

By our above observation, there is $c_1\in \nu_{Y_1}(c)$ such that $c+d_1=c_1+d_2$. It follows that $d_1-d_2=c_1-c\in \nu_{Y_1}(c)-c$.
By what we just showed, $\nu_{Y_1}(c)-c=\nu_{Y_2}(d)-d$, hence $d_1-d_2\in \nu_{Y_2}(d)-d$.\end{proof}

We are now ready to construct the maximal infinitesimal subgroup of $(\CF,+)$.
\begin{proposition}\label{P:infinit}
Let $Y\subseteq \CF$ be strongly internal to $K$ over $A$, of maximal dp-rank $n$.

\begin{enumerate}
    \item Let  $d\in Y$ be with $\dpr(d/A)=n$. Then $\nu_{Y}(d)-d$ is a type definable subgroup of $(\CF,+)$. Moreover, it is independent of the choice of $d$, and we denote it $\nu_{Y}$.

    \item For any strongly internal $Y_0\subset \CF$,  with $\dpr(Y_0)=n$, $\nu_Y=\nu_{Y_0}$.

    %As a result, we may write $\nu$ for $\nu_Y$.
\end{enumerate}
\end{proposition}
\begin{proof}
(1) We  return to our one dimensional $I\sub \CF$ which is strongly internal to $K$ (Corollary \ref{C:existence of I}). We fix any $b\in I$
and $c\in Y$ such that $\dpr(b,c,d/A)=2n+1$ and apply  Lemma \ref{L:generic linearity for K}.

We obtain $A\subseteq B$-definable $J\sub I$ and $S\sub Y\times Y$ containing $\la c,d\ra$,  with  $\dpr(J)=1$, $\dpr(S)=2n$, and $B\ni b$,  such that the map $(x-b)y+z$ sends $I\times S$ into $Y$ and such that $\dpr(c,d/B)=2n$. Because $\la c,d\ra \in S\subseteq Y\times Y$ has maximal dp-rank, it follows by Lemma \ref{L:generic smaller ball} that there exist $B\subseteq B_0$-definable $\tau_{Y}$-open definable sets $Y_1,Y_2\sub Y$, neighborhoods of $c$ and $d$, respectively, such that $Y_1\times Y_2\sub S$ and $\dpr(c,d/B_0)=\dpr(c,d/B)=2n$.

Since $\dpr(J\times S)=2n+1$, we can find $b_1\in J$ with $\dpr(b_1,c,d/B_0)=2n+1$. Let $Y_1'=(b_1-b)Y_1$ and $c'=(b-b_1)c$.
By our assumptions, $Y_1'+Y_2\sub Y$ and $\dpr(Y_1')=\dpr(Y_2)=\dpr(Y)$. Therefore,  the assumptions of Lemma \ref{addition closure} are satisfied, with $\dpr(c',d/b_1B_0)=2n$. Thus, $\nu_{Y_2}(d)-d$ is a subgroup of $(\CF,+)$, equal to $\nu_{Y_1'}(c')-c'$.
By Lemma \ref{internal-locally the same}, $\nu_{Y_2}(d)=\nu_Y(d)$ and $\nu_{Y_1'}(c)-c=\nu_{(b-b_1)Y}(c')-c'$. Thus, $\nu_Y(d)-d=\nu_{(b-b_1)Y}(c')-c'$ is a subgroup of $(\CF,+)$.

If we now have $d'\in Y$ such that $\dpr(d'/A)=n$, then we choose $c\in Y$ with $\dpr(c,d/A)=\dpr(c,d'/A)=2n$, and then $b_1\in J$ such that $\dpr(b_1,c,d/B_0)=\dpr(b_1,c,d'/B_0)=2n+1$. Repeating the above argument we conclude that
\[\nu_Y(d)-d=\nu_{(b-b_1)Y}(c')-c'=\nu_Y(d')-d'.\]

(2) Let $Y_0\subseteq \CF$ be any set strongly internal to $K$ with $\dpr(Y_0)=n$ and let $Y'=Y\cup Y_0$. It follows from   Lemma \ref{internal-locally the same} that $\nu_Y=\nu_{Y'}=\nu_{Y_0}$.
\end{proof}

\begin{corollary}\label{C:invariant under mult}
The partial type $\nu$ is invariant under multiplication by scalars from $\CF$.
\end{corollary}
\begin{proof}
    Let $c\in \CF$ and let $Y\subseteq \CF$ be any definable subset strongly internal to $K$ over $\0$, of maximal dp-rank $n$.
    Let $d\in Y$ be such that $\dpr(d/c)=n$, so also $\dpr(c\cdot d)=n$. The function $x\mapsto cx$ sends $Y$ to $cY$, and by lemma \ref{function}, it sends $\nu_Y(d)$ onto $\nu_{cY}(cd)$. Hence,
    $$c\nu=c(\nu_Y(d)-d)=c\nu_Y(d)-cd=\nu_{cY}(cd)-cd=\nu.$$\end{proof}

\section{Embedding the field $\CF$ into $M_n(K)$}
Assume that $K$ is a P-minimal field. 
In the present section we prove that $\CF$ is in definable isomorphism with a finite extension of $K$. We do so by identifying $\CF$ with a subfield of $M_n(K)$. Using dp-minimality  of $K$, we show that $\CF$ must be a finite extension of the canonical embedding of $K$ into $M_n(K)$.

In order to embed $\CF$ into the ring of matrices we need to endow the additive subgroup $\nu$ introduced in the previous section,
with a $\CD$-structure with respect to $K$. That is, we will see that $\nu$ is an open set, and that group operations are differentiable. Recall,

%\begin{equation*}\tag{A}
%\parbox{\dimexpr\linewidth-4em}{
%For every definable $f:D\to K$, $D\subseteq K^n$ with non-empty interior, there exists an open subset $U\subseteq K^n$ such that $f\restriction U\cap D$ is $\CC^1$. }
%\end{equation*}

\begin{definition}
Given $U\subseteq K^n$ open, a map $f:U\to K^m$, is \emph{differentiable} at $x_0\in U$ if there exists a linear map $D_{x_0}f:K^n\to K^m$ such that 
\[\lim_{x\to x_0}\frac{f(x)-f(x_0)-D_{x_0}f(x-x_0)}{x-x_0}=0.\]
\end{definition}

If $f$ is differentiable at a point $x_0$ we  say that $x_0$ is a $\CD^1$-point of $f$ and that $f$ is $\CD^1$ at $x_0$.
It is easy to see that if $f$ is definable then set of $x_0$'s at which $f$ is $\CD^1$ is definable as well. 
For the purposes of the present section we assume:

\begin{equation*}\tag{A}
\parbox{\dimexpr\linewidth-4em}{
For every definable open $D\subseteq K^n$, and definable $f:D\to K$,  there exists an open subset $U\subseteq K^n$ such that $f\restriction U\cap D$ is \emph{differentiable}.}
\end{equation*}

\begin{remark}
It is standard to see that (A) implies that for any such $\emptyset$-definable $f:D\to K$, any $d\in D$ with $\dpr(D)=\dpr(d)$ is a $\CD^1$-point of $f$.
\end{remark}

By \cite[Proposition 4.6]{SimWal} if $S\sub K^n$ has dimension $d$ then there exists an open set $U\sub K^n$ and a projection $\pi:K^n \to K^d$ such that $\pi(U)$ is open and  $\pi: (U\cap S)\to K^d$ is a homeomorphism onto its image. Thus, if $Y\sub \CF$ is strongly internal to $K$ with $\dpr(Y)=n$ we can find $Y_0\sub Y$ and a definable injection $g_0: Y_0\to K^n$ with open image.  This is the setting for the following proof, due to Marikova \cite{MarikovaGps}, who proved it for continuity in the o-minimal context. The exact same proof goes through for any property that is generically true for all definable functions.
\begin{proposition}\label{P:C1 structure}
    Let $Y\subseteq \CF$ be a definable subset with dp-rank $n$, $g:Y\to K^n$ an injective $A$-definable map with open image, and let  $\CK$ be  a $|K|^+$-saturated elementary extension of $K$. Then there exists $a\in Y$, with $\dpr(a/A)=\dpr(Y)$, such that,  
    \begin{enumerate}
        \item The map $ x\mapsto g(x+a)$ induces on $\nu(\CK)$ a structure of a $\CD^1$-group, namely, the group operations are $\CD^1$ when read via $g$.
        \item Any $K$-definable endomorphism $\alpha:(\CF(\CK),+)\to (\CF(\CK),+)$ sending $\nu(\CK)$ to itself is a $\CD^1$-map with respect to the above differential structure on $\nu$.
    \end{enumerate}
\end{proposition}
% \begin{remark}
%     Assume that $Y\subseteq \CF$ is a definable subset and $g:Y\to K^r$ is an injective map. If $\dpr(Y)=n$ then there exists an open subset $U\subseteq K^r$ and a map $\pi:g(Y)\cap U\to K^n$ with open image that is an homeomorphism onto its image. Thus the map $\pi\circ g$ satisfies the assumptions of the proposition.
% \end{remark}
\begin{proof}
%    Let $Y$, $g$ and $d$ be as in the statement.
All group operations appearing in the proof are the restriction to $Y$ of $\CF$ addition (and subtraction).

Let $d\in Y$ be such that $\dpr(d/A)=\dpr(Y)$. By replacing $Y$ with $Y-d$, we may assume that $\nu \vdash Y$. Now absorb $A$ and $d$ into the language. In order to keep notation simple we identify $Y$ with its image under $g$ (so $g=\id$). \\
%Since $\nu$ is a subgroup, we may assume, using compactness and shrinking $Y$ if needed, that $Y=-Y$ and that $Y'+Y'\sub Y$ for some definable $Y'$ of full rank.

\noindent (1)
Note that $\nu=\{U\subseteq K^n: U \text{ is $K$-open and } 0_{\mathcal{F}}\in U\}$ and thus $\nu(\CK)$ is open in $\CK^n$.

%As then for $x,y\in \nu$ we get, denoting $h(x)=g(x+d)$:
%\[
 %   h(x+y)=h(x)+h(y)-d
%\] is $\CD^1$, showing that $\CF$-addition is $\CC^1$ on $\nu$. Similarly, $d-x=h(-x)+d-d$ is $\CC^1$ with the desired conclusion. So, as being $\CC^1$ if first order, in order to prove the first clause of the proposition , by compactness, it suffices to show:

%and therefore -- by definition and the assumption that $g(Y)$ is open --  $g(x+d)$ maps $\nu$ to $\bigcap\{U(\CK): d\in U(\CK)\subseteq \CK^n \text{ open $K$-definable}\}$. Since it is an open subset of $\CK^n$ it inherits the obvious $C^1$-manifold structure.

    %\begin{claim}
     %   The function $\la x,y,z\ra \mapsto x-y+z$ is $\CC^1$ at $\la d,d,d\ra$.
    %\end{claim}
    %\begin{claimproof}
%        We work in $\CK$. Note that $\dpr(\nu)=n$ and so we can find $a\in \nu(\CK)$ with $\dpr(a/K)=n$. Let $c\in \nu(\CK)$ be such that $\dpr(c,a/K)=2n$.
        Since $\nu$ is a group, type definable over $K$, there are, by compactness, $K$-definable open sets $V_1, V_0$, such that $\nu\vdash V_1\subseteq V_0\sub K^n$ and
        %Choose any $V_0(\CK)=g^{-1}(U_0(\CK))-d$, where $g(d)\in U_0(\CK)\subseteq \CK^n$ is an $K$-definable open set. By compactness, there exists $K$-definable open $g(d)\in U_1(\CK)\subseteq U_0(\CK)$  such that for $V_1(\CK)=g^{-1}(U_1(\CK))-d$,
        \[\phi_4:\la x,y\ra \mapsto x+y : V_1^2\to V_0.\]
        Similarly, we find $V_2\subseteq V_1$ such that
        \[\phi_3: \la x,y,z\ra\mapsto \la -z,x+y\ra : V_2^2\to V_1^2.\] We also find
        $V_3\subseteq V_2$ such that
        \[\phi_2: \la x,y,z\ra\mapsto \la x+y,z\ra : V_3^3\to V_2^2,\]
        and $V_4\subseteq V_3$ such that
        \[\phi_1: \la x,y,z,w\ra\mapsto \la w+x,-y,z\ra : V_4^3\to V_3^3.\]
        We may assume that all the above are $\emptyset$-definable. Let $a,b\in V_4$ with $\dpr(a,b)=2\dpr(Y)=2\dpr(V_4)$. By assumption (A), $\varphi_1(x,y,z,b)$ is $\CD^1$ at $(a,a,a)$, $\varphi_2$ is $\CD^1$ at $(b+a,-a,a)$, $\varphi_3(x,y,b)$ at $(b,a)$ and $\varphi_4$ at $(b^{-1},b+a)$. Composing, we obtain $\varphi_4\circ\varphi_3\circ\varphi_2\circ\varphi_1(x,y,z)=x-y+z$, so we have shown that the map $\la x,y,z\ra \mapsto x-y+z$ is $\CD^1$ at $(a,a,a)$. 
        %  \vspace{.2cm}
        %  \noindent{\bf Claim} The map $x-y+z$ is $\CD^1$ at $(a,a,a)$.
        %  \vspace{.2cm}
        In fact, the proof actually provides an open set, $U\sub V_4$ which, by compactness, we may take to be $K$-definable, with $a\in U$, such that $\la x,y,z\ra\mapsto x-y+z:U^3\to V_0$ is $\CD^1$.

    Our goal is to show that the push-forward of $+\restriction \nu^2$ and $-()\restriction\nu$ under the map $x\mapsto x+a$ are $\CD^1$ in the sense of $K$. Namely, we need to prove that the functions
    $(x-a)+(y-a)+a$ and $-(x-a)+a$ on $\nu_Y(a)^2$ and $\nu_Y(a)$, respectively, are  $\CD^1$. Both follow immediately from our choice of $U$. \\

    %    By Proposition \ref{P:infinit}, $\nu=\nu_{Y}(a)-a$. % Since $\nu_{Y}(a)\vdash U$ we may now conclude that %the function mapping $x-a,y-a\in \nu_{Y}-a$ to %$(x-a)+(y-a)+a=x-a+y$ is $\CD^1$ (note $x+y-a\in \nu$). %Similarly, we see that $-(x-a)+a$ is $\CD^1$ by %considering the map $\la x-a\ra\mapsto -(x-a)+a=a-x+a$.
        % By (A), since $\dpr(-c,c+a/K)=2n$, we can find $(-c,c+a)\in V_1'(\CK)\subseteq V_1(\CK)^2$ (with $V_1'(\CK)+d$ open) such that $\phi_4\restriction V_1'(\CK)$ is $C^1$. We do the same for $\phi_3$ and $\la c,a\ra$ and $\la c,a\ra\in V_2'(\CK)$. Finally we find $a \in V_4'(\CK)\subseteq V_4(\CK)$ such that $\phi_1\restriction V_4'(\CK)^3$ is $C^1$.

        % Composing $\phi_4\circ \phi_3\circ \circ \phi_2\circ \phi_1\restriction V_4'(\CK)^3: \la x,y,z\ra\mapsto x-y+z$ is a $C^1$ map into $V_0(\CK)$.

        % Now, since $K\prec K$, we may $K$ open subsets $g(d)\in U\subseteq V\subseteq K^n$ such that $\la x,y,z\ra \mapsto x-y+z$ maps $g^{-1}(U)-d\to g^{-1}(V)-d$ and is $C^1.$
    %\end{claimproof}
        % Let $x-d,y-d\in \nu_{Y,g}(d)-d\subseteq g^{-1}(U)(\CK)-d$. Thus $\la x-d,y-d\ra\mapsto (x-d)-(d-d)+(y-d)=(x-d)+(y-d)$ is $C^1$. Likewise $(x-d)\mapsto (d-d) -(x-d)+(d-d)=-(x-d)$ is $C^1$.

        \noindent (2) 
        %By compactness, there exist $K$-definable open subsets of $Y$, $\nu\vdash U\subseteq V$ such that $\alpha(x)\in V$ for all $x\in U$. 
        Let $e\in \nu(\CK)$ with $\dpr(e/K)=\dpr(\nu)=n$. By assumption $(A)$, $e$ is a $\CD^1$-point of $\alpha$. Since $\alpha$ is a homomorphism and $\nu$ is a $\CD^1$-group, $\alpha$ is a $\CD^1$-function on $\nu$. %$\alpha(x)=\alpha(x+e)-\alpha(e)$ for any $x\in \nu(\CK)$, and $\nu(\CK)$ is a $\CD^1$-group, $\alpha$ is $\CD^1$ at $x$ as well.
        % assumption on $\alpha$, we may find by compactness two $K$-open set $V\subseteq U$ such that $\alpha$ is defined on $c\in V(\CK)$ and maps into $U(\CK)$. By (A), we may find a definable open subset $c\in V_0(\CK)\subseteq U(\CK)$ such that $\alpha\restriction V_0(\CK)$  is $C^1$. In particular, $\alpha$ is $C^1$ at $c$. But since $\alpha$ is a homomorphism and $\nu$ is a $C^1$-group, $\alpha$ is $C^1$ at every point of $\nu$.
\end{proof}

We need the following easy and well known fact:
\begin{remark}
    If $K$ is a dp-minimal field, then $K$ has no definable infinite subfields. Indeed, if $L$ were such a field then $L$ itself is dp-minimal. And if we had some $v\in K\setminus L$ we could define $T:L^2\to K$ by $(a,b)\mapsto a+bv$. Since $v$ is $L$-linearly independent of $1$, we get that $T$ is a linear injection, so $\dpr(T(L^2))=2$ which is impossible.
\end{remark}

\begin{proposition}
        The field $\CF$ is definably isomorphic to a finite extension of $K$.
\end{proposition}
    \begin{proof}
        Let $\nu$ be as before, let $Y\subseteq \CF$ be strongly internal of maximal dp-rank $n$ and assume that $g:Y\to K^n$ is the corresponding injective map.

        By Corollary \ref{C:invariant under mult}, for any $z\in \CF$ the function $\lambda_z: x\mapsto zx$ leaves the type $\nu$ invariant. By Proposition \ref{P:C1 structure}, $\lambda_z$ is a $\CD^1$ map on $\nu(\CK)$. 
        Thus, for every $z\in \CF$ there exists a definable neighborhood $V\ni 0_{\CF}$ such that $\lambda_z$ is $\CD^1$ on $V$.
        This is a first order property which holds in $\CK$ as well.
        Consequently, to each $z\in \CF$ may associate, definably, the Jacobian matrix $J_z\in M_n(K)$ of $\lambda_z$ at $0$.

        As in the proof of \cite[Lemma 4.3]{OtPePi}, an application of the chain rule \cite[Remark 4.1.ii]{p-adiclie}, implies that the map $z\mapsto J_z$ is a ring homomorphism  sending $1\in \CF$ to the identity matrix $I_n$. Since $\CF$ is a field, the map is injective so we may embed $\CF$ into a definable subring of $M_n(K)$.

        We may now view $\CF$ as a definable subfield of $M_n(K)$. Let $K_0=\{xI_n:x\in K\}$, where now we take the usual scalar multiplication in the algebra of matrices. Thus $K_0\cap F$ is an infinite definable subfield of $K$ (as they both contain $I_n$ and are of characteristic $0$). Since $K_0\cong K$ is dp-minimal, it has no infinite definable subfield, so $K_0\cap \CF=K_0$ i.e. $K_0\subseteq \CF$. Thus $\CF$ is a finite extension of $K_0$.
    \end{proof}

%Recall that a strictly $P$-minimal field is a field elementarilly equivalent to a $P$-minimal expansion of a finite field extension of $\Qq_p$. By \cite[Theorem 1.9]{KuiLee} (combined with generic continuity for $P$-minimal fields) any definable function in a strictly $P$-minimal field is $\CC^1$ away from a finite subset of its domain. This immediately gives assumption (A) for strictly $P$-minimal fields. Indeed, if $K$ is strictly $P$-minimal $U\subseteq K^n$ is open and $f: U\to K$ is a definable function, then for any  $a=(a_1,\dots, a_n)\in U$ the function $f_{a,1}:=f(x,a_2,\dots, a_n)$ is smooth on a cofinite subset of $U_{a_2,\dots, a_n}$. In particular, if $a$ is generic then $f_{a,i}$ is smooth at $a_1$. Since the same is true for all $i$, we get that $f$ is smooth at $a$, for any generic $a$. Thus we get:

Summing all the above we get:

\begin{theorem}\label{T:main thm}
    Let $K$ be a $P$-minimal field $K$. If for every definable $f:D\to K$, $D\subseteq K^n$ with non-empty interior, there exists an open subset $U\subseteq K^n$ such that $f\restriction U\cap D$ is differentiable then every infinite interpretable field is definably isomorphic to a finite extension of $K$.
\end{theorem}

Let $\CL^{an}$ be the subanalytic language for the p-adics, see \cite{DenvdDr}.
For any prime $p$, let $\mathbb{Q}_p^{an}$ be the p-adic field in the subanalytic language.

\begin{corollary}\label{C:cor for Qpan}
Let $(K,v)$ be a valued field. Assume that either
\begin{enumerate}
    \item $K$ is elementary equivalent to $\mathbb{Q}_p^{an}$ in the subanalytic language $\CL^{an}$ or
    \item $K$ is p-adically closed.
\end{enumerate}
Then every infinite  field interpretable in $K$ is definably isomorphic to a finite extension of $K$.
\end{corollary}
\begin{proof}
    In case of (1), $K$ is $P$-minimal by \cite[Theorem B]{vdDriesHasMac}). As a result it is sufficient to verify assumption (A). Since having non-empty interior is definable in families and being differentiable is definable (using the parameters needed to define the function), it is enough to check assumption (A) for $\emptyset$-definable functions whose domain has non-empty interior and assume that either $f$ is a definable function in $\mathbb{Q}_p^{an}$ or in a finite extension of $\mathbb{Q}_p$ (in the valued field language). Since every analytic map is differentiable \cite[Proposition 6.1]{p-adiclie}, the result follows readily from \cite[Proposition 3.29]{DenvdDr} in case (1) and \cite[Theorem 1.1]{SccdDries} in case (2)\footnote{See, also, the discussion in \cite[Section 5]{SccdDries}}.
\end{proof}

% {\bf I did not check the references below so i cannot judge on that. Also, we never used $Kv$ for the residue field.

% Also, one needs to check how exactly the definability of $\CF$ rather than its interpretability plays against the other assumptions in previous section. There was a vague comment there which i could not figure out}

% We conclude with another generalisation of Pillay's theorem, \cite{PilQp}:
% \begin{theorem}
%     If $K$ is a dp-minimal field of characteristic 0, and $\CF$ is an infinite field definable in $K$ in the language of rings, then $\CF$ is definably isomorphic to a finite extension of $K$.
% \end{theorem}
% \begin{proof}
%     For $K$ algebraically closed this goes back to Poizat, \cite{PoiFields} and for $K$ real closed this is due to Otero, Pillay and the third author, \cite{OtPePi}. So we may assume that $K$ is neither real closed nor algerbaically closed. Thus, by \cite{JohnDpJML}, $K$ is a henselian field with respect to a definable valuation $v$, and $K$ satisfies all the assumptions of \cite{SimWal}, as spelled out in Remark \ref{R:definable}. So by the conclusion of that remark, it will suffice to show that $K$ satisfies condition (A) above. This follows from  \cite[Theorem 7.1.6, Theorem 5.15]{CluHalRid} if $\mathrm{char}(Kv)=0$ is and from the Theorem 7.1.7 and Proposition 6.1.5 if $\mathrm{char}(Kv)>0$.
%     \end{proof}

 \bibliographystyle{plain}
\bibliography{Bibfiles/harvard}

\begin{thebibliography}{10}

\bibitem{DenvdDr}
J.~Denef and L.~van~den Dries.
\newblock {$p$}-adic and real subanalytic sets.
\newblock {\em Ann. of Math. (2)}, 128(1):79--138, 1988.

\bibitem{DolGodViscerality}
Alfred {Dolich} and John {Goodrick}.
\newblock {Tame topology over definable uniform structures: viscerality and
  dp-minimality}.
\newblock {\em arXiv e-prints}, page arXiv:1505.06455, May 2015.

\bibitem{DoGoStrong}
Alfred Dolich and John Goodrick.
\newblock Strong theories of ordered {A}belian groups.
\newblock {\em Fund. Math.}, 236(3):269--296, 2017.

\bibitem{HasMac}
Deirdre Haskell and Dugald Macpherson.
\newblock A version of o-minimality for the {$p$}-adics.
\newblock {\em J. Symbolic Logic}, 62(4):1075--1092, 1997.

\bibitem{HaPetRCVF}
Assaf {Hasson} and Ya'acov {Peterzil}.
\newblock {Interpretable fields in real closed valued fields and some
  expansions}.
\newblock {\em arXiv e-prints}, page arXiv:2102.00814, February 2021.

\bibitem{MarikovaGps}
Jana Ma\v{r}\'{\i}kov\'{a}.
\newblock Type-definable and invariant groups in o-minimal structures.
\newblock {\em J. Symbolic Logic}, 72(1):67--80, 2007.

\bibitem{OtPePi}
Margarita Otero, Ya'acov Peterzil, and Anand Pillay.
\newblock On groups and rings definable in o-minimal expansions of real closed
  fields.
\newblock {\em Bull. London Math. Soc.}, 28(1):7--14, 1996.

\bibitem{PilQp}
Anand Pillay.
\newblock On fields definable in {${\bf Q}_p$}.
\newblock {\em Arch. Math. Logic}, 29(1):1--7, 1989.

\bibitem{p-adiclie}
Peter Schneider.
\newblock {\em {$p$}-adic {L}ie groups}, volume 344 of {\em Grundlehren der
  Mathematischen Wissenschaften [Fundamental Principles of Mathematical
  Sciences]}.
\newblock Springer, Heidelberg, 2011.

\bibitem{SccdDries}
Philip Scowcroft and Lou van~den Dries.
\newblock On the structure of semialgebraic sets over {$p$}-adic fields.
\newblock {\em J. Symbolic Logic}, 53(4):1138--1164, 1988.

\bibitem{SimDPMinOrd}
Pierre Simon.
\newblock On dp-minimal ordered structures.
\newblock {\em J. Symbolic Logic}, 76(2):448--460, 2011.

\bibitem{Simdpr}
Pierre Simon.
\newblock Dp-minimality: invariant types and dp-rank.
\newblock {\em J. Symb. Log.}, 79(4):1025--1045, 2014.

\bibitem{SimWal}
Pierre Simon and Erik Walsberg.
\newblock Tame topology over dp-minimal structures.
\newblock {\em Notre Dame J. Form. Log.}, 60(1):61--76, 2019.

\bibitem{vdDriesHasMac}
Lou van~den Dries, Deirdre Haskell, and Dugald Macpherson.
\newblock One-dimensional {$p$}-adic subanalytic sets.
\newblock {\em J. London Math. Soc. (2)}, 59(1):1--20, 1999.

\bibitem{WalTrace}
Erik {Walsberg}.
\newblock {Notes on trace equivalence}.
\newblock {\em arXiv e-prints}, page arXiv:2101.12194, January 2021.

\end{thebibliography}
\end{document}